\documentclass{ws-ijgmmp}

\begin{document}

\markboth{H. Attarchi and M. M. Rezaii}
{Cartan Spaces and Natural Foliations on the Cotangent Bundle}

%
\catchline{}{}{}{}{}
%

\title{Cartan Spaces and Natural Foliations on the Cotangent Bundle}

\author{H. Attarchi}

\address{Department of Mathematics and Computer Science\\
Amirkabir University of Technology, Tehran, Iran\\
\email{hassan.attarchi@aut.ac.ir}}

\author{M. M. Rezaii}

\address{Department of Mathematics and Computer Science\\
Amirkabir University of Technology, Tehran, Iran\\
\email{mmreza@aut.ac.ir}}

\maketitle

\begin{history}
\received{(Day Month Year)}
\revised{(Day Month Year)}
\end{history}

\begin{abstract}
In this paper, the natural foliations in cotangent bundle $T^*M$ of Cartan space $(M,K)$ is studied. It is shown that geometry of these foliations are closely related to the geometry of the Cartan space $(M,K)$ itself. This approach is used to obtain new characterizations of Cartan spaces with negative constant curvature.
\end{abstract}

\keywords{Cartan Spaces; foliation; Cartan space of negative constant curvature.}

\section{Introduction}
Lagrange space has been certified as an excellent model for some important problems in Relativity, Gauge Theory and Electromagnetism~\cite{anastasia,anastasia1}. The geometry of Lagrange spaces gives a model for both the gravitational and electromagnetic field. P. Finsler in his Ph.D. thesis introduced the concept of general metric function, which can be studied by means of variational calculus. Later, L. Berwald, J.L. Synge and E. Cartan precisely gave the correct
definition of a Finsler space~\cite{matsumoto,miron}. These structures play a fundamental role in study of the geometry of tangent bundle $TM$. The geometry of cotangent bundle $T^*M$ and tangent bundle $TM$ which follows the same outlines are related by Legendre transformation. From this duality, the geometry of a Hamilton space can be obtained from that of certain Lagrange space and vice versa. As a particular case, we can associate to a given Finsler space its dual, which is a Cartan space~\cite{miron1,miron2}. Using this duality several important results in the Cartan spaces can be obtained: the canonical nonlinear connection, the canonical metrical connection etc. Therefore, the theory of Cartan spaces has the same symmetry and beauty like Finsler geometry. Moreover, it gives a geometrical framework for the Hamiltonian theory of Mechanics or Physical fields. With respect to the importance of these spaces in Physical areas and the quick growth of Finsler geometry in recent decades, present work is formed to develop some properties of Finsler geometry to the Cartan spaces. In~\cite{bejancu} and~\cite{bejancu2008} the natural foliations on tangent bundle of a Finsler manifold have been studied. Here, some theorems of these objects are reconsidered on cotangent bundle of a Cartan space for some natural foliations such as Liouville-Hamilton vector field and its complement distribution in $TT^*M$.

Let $(M,K)$ be a Cartan space then to achieve these aims, the present work is organized in the following way. In Section 2, some definitions and results on Cartan spaces which is needed in following sections is provided. In particular, canonical nonlinear connections, Sasaki lift of the metric on cotangent bundle and coefficients of Levi-Civita connection are presented, for more details see~\cite{miron}. In Section 3, to simplify some equations and proofs in following sections a new frame is set on $TT^*M$ such that $TT^*M$ is decomposed to two foliations Liouville-Hamilton vector field and its complement distribution in $TT^*M$. The Levi-Civita connection on a Cartan space is computed in the new basis and relation of curvature tensor fields of level hypersurfaces $K=const.$ and $T^*M$ is calculated. In Section 4, the ideas of~\cite{bejancu} are developed to cotangent bundle of a Cartan space and six natural foliations of cotangent bundle are introduced. They are studied from different viewpoints such as being totally geodesic, bundle like and etc. As a main result of present paper, the condition of being a Cartan space with negative constant curvature is found. Finally in Section 5, \emph{1}-indicatrix bundles in cotangent bundle of the Cartan space $(M,K)$ which is denoted by $I^*\!M(1)$ is studied. Similar to Finsler geometry~\cite{bejancu2008}, it is shown that it naturally has contact structure. In addition, it is shown that $I^*\!M(1)$ cannot have Sasakian structure with lifted Sasaki metric $G$.

\section{Preliminaries and notations}
Let $M$ be a real \emph{n}-dimensional differentiable manifold and let $(T^*M,\pi^*,M)$ be its cotangent bundle. If $(x^i)$, $(i=1,...,n)$, is a local coordinate system on a domain $U$ of a chart on $M$, the induced system of coordinates on $\pi^{*-1}(U)$ are $(x^i,p_i)$. The coordinates $p_i$ are called \emph{momentum variables}. The \emph{Liouville-Hamilton vector field}, \emph{Liouville 1-form} and \emph{canonical symplectic structure} on $T^*M$ are denoted by $C^*$, $\omega$ and $\theta$, respectively, and their local expressions are as follows:
\begin{equation}~\label{geo object}
C^*=p_i\partial^i, \ \ \omega=p_idx^i, \ \ \theta=dp_i\wedge dx^i
\end{equation}
where $\partial^i:=\frac{\partial}{\partial P_i}$. Let $\{.,.\}$ be the Poisson bracket on $T^*M$, defined by:
$$\{f,g\}=\frac{\partial f}{\partial p_i}\frac{\partial g}{\partial x^i}-\frac{\partial g}{\partial p_i}\frac{\partial f}{\partial x^i},\ \ \forall f,g\in C^{\infty}(T^*M)$$
A Cartan space is a pair $(M,K(x,p))$ such that following axioms hold:
\begin{enumerate}
\item{K is a real positive function on $T^*M$, differentiable on $T^*M\backslash\{0\}$ and continuous on the null section of the projection $\pi^*$.}
\item{K is positively 1-homogeneous with respect to the momenta $p_i$.}
\item{The Hessian of $K^2$, with elements $g^{ij}(x,y)=\frac{1}{2}\partial^i\partial^jK^2$ is positive-defined.}
\end{enumerate}
The nonlinear connection coefficient $N_{ij}$ of the Cartan space $(M,K)$ is given by:
$$N_{ij}=\frac{1}{4}\{g_{ij},K^2\}-\frac{1}{4}\left(g_{ik}\frac{\partial^2K^2}{\partial p_k\partial x^j}+g_{jk}\frac{\partial^2K^2}{\partial p_k\partial x^i}\right)$$
where $(g_{ij})$ is the inverse matrix of $(g^{ij})$. The adapted basis of $TT^*M$ and $T^*T^*M$ with respect to the natural nonlinear connection coefficient $N_{ij}$ are expressed as follows:
\begin{equation}~\label{basis}
\left\{
\begin{array}{l}
TT^*M=<\frac{\delta}{\delta x^i},\partial^i>\ ,\ \ \ \frac{\delta}{\delta x^i}=\frac{\partial}{\partial x^i}+N_{ij}\partial^j\\ \cr T^*T^*M=<dx^i,\delta p_i>\ ,\ \ \ \delta p_i=dp_i-N_{ij}dx^j
\end{array}
\right.
\end{equation}
In addition, with respect to these bases the Sasakian lift of the metric tensor $g^{ij}$ on $T^*M$ is shown by $G$ and given by:
\begin{equation}~\label{metric}
G:=g_{ij}dx^i\otimes dx^j+g^{ij}\delta p_i\otimes\delta p_j
\end{equation}
The almost complex structure compatible with metric $G$ is shown by $J$ and has local expression as follows:
$$J:=g^{ij}\frac{\delta}{\delta x^j}\otimes\delta p_i-g_{ij}\partial^j\otimes dx^i$$
Then by direct calculations and using~(\ref{basis}) it is obtained that:
\begin{equation}~\label{brackets}
[\frac{\delta}{\delta x^i},\frac{\delta}{\delta x^j}]=R_{ijk}\partial^k,\ \ \ \ [\partial^j,\frac{\delta}{\delta x^i}]=N_{ik}^j\partial^k
\end{equation}
where $R_{ijk}=\frac{\delta N_{jk}}{\delta x^i}-\frac{\delta N_{ik}}{\delta x^j}$ and $N_{ik}^j=\partial^jN_{ik}$. Now, let $\nabla$ be the Levi-Civita connection on Riemannian manifold $(T^*M,G)$. Then we can prove the following.
\begin{theorem}~\label{levi1}
Let $(M,K)$ be a Cartan space. Then the Levi-Civita connection $\nabla$ on $(T^*M,G)$ is locally expressed as follows:
$$
\left\{
\begin{array}{l}
\nabla_{\frac{\delta}{\delta x^i}}\frac{\delta}{\delta x^j}=\Gamma_{ij}^k\frac{\delta}{\delta x^k}+\frac{1}{2}(R_{ijk}+g_{ijk})\partial^k \cr
\nabla_{\frac{\delta}{\delta x^i}}\partial^j=\nabla_{\partial^j}\frac{\delta}{\delta x^i}-N_{ih}^j\partial^h=-\frac{1}{2}(g_i^{jk}+R_{ish}g^{hj}g^{sk})\frac{\delta}{\delta x^k}\\
+\frac{1}{2}(\frac{\delta g^{jk}}{\delta x^i}+\partial^k(N_{is}g^{sj})-\partial^j(N_{is}g^{sk}))g_{kh}\partial^h \cr
\nabla_{\partial^i}\partial^j=-\frac{1}{2}(\frac{\delta g^{ij}}{\delta x^k}+N_{ks}^ig^{sj}+N_{ks}^jg^{si})g^{kh}\frac{\delta}{\delta x^h}+\frac{1}{2}g_k^{ij}\partial^k
\end{array}
\right.
$$
where
$$\Gamma_{ij}^k=\frac{g^{kh}}{2}\{\frac{\delta g_{ih}}{\delta x^j}+\frac{\delta g_{jh}}{\delta x^i}-\frac{\delta g_{ij}}{\delta x^h}\}$$
$$g_{ijk}=g_{is}g_{jt}g_k^{st}=g_{is}g_{jt}g_{kh}g^{sth}=g_{is}g_{jt}g_{kh}\partial^h(g^{st})$$
\end{theorem}
\begin{proof}
By using~(\ref{brackets}), definition of Levi-Civita connection
$$\left\{\begin{array}{l}
2G(\nabla_XY,Z)=XG(Y,Z)+YG(X,Z)-ZG(X,Y)\cr
-G([X,Z],Y)-G([Y,Z],X)+G([X,Y],Z)
\end{array}
\right.$$
and a straightforward calculation the proof is complete. \end{proof}

\section{Adapted frame on indicatrix bundle of a Cartan space}
Supposed that $(M,K)$ is a \emph{n}-dimensional Cartan space. In natural way, the vertical and horizontal distributions of $TT^*M$ are given by:
$$VT^*M=<\partial^1,...,\partial^n>\ ,\ \ \ HT^*M=<\frac{\delta}{\delta x^1},...,\frac{\delta}{\delta x^n}>$$
Consider the Riemannian manifold $(T^*M,G)$, then, orthogonal distribution to Liouville-Hamilton vector field $C^*$ in vertical distribution $VT^*M$ is denoted by $V'T^*M$. By definition of $V'T^*M$, it is easy to see that $V'T^*M$ is a foliation in $TT^*M$. Therefore, there is the local chart $(U,\varphi)$ on $T^*M$ such that:
$$TT^*M|_U=<\bar{\partial}^1,...,\bar{\partial}^{n-1},v_1,...,v_{n+1}>$$
where
$$V'T^*M|_U=<\bar{\partial}^1,...,\bar{\partial}^{n-1}>$$
It is obvious that $\bar{\partial}^a$ for all $a=1,...,n-1$ are vector fields in $VT^*M$ and they can be written as a linear combination of the natural basis $\{\partial^1,...,\partial^n\}$ of $VT^*M$ as follows:
$$\bar{\partial}^a=E_i^a\partial^i\ \ \ \forall a=1,...,n-1$$
where $E_i^a$ is the $(n-1)\times n$ matrix of maximum rank. The first property of this matrix is $E_i^ag^{ij}p_j=0$. Moreover, on $U$ the vertical distribution $VT^*M$ is locally spanned by:
\begin{equation}~\label{basis2}
\{C^*,\bar{\partial}^1,...,\bar{\partial}^{n-1}\}
\end{equation}
Then the almost complex structure $J$ acts on the basis~(\ref{basis2}) as follows:
$$J(\bar{\partial}^a)=E_j^ag^{ji}\frac{\delta}{\delta x^i}\ ,\ \ \ \xi=J(C^*)$$
We put $\bar{E}_a^i=E_j^bg_{ba}g^{ji}$ where $g^{ab}:=G(\bar{\partial}^a,\bar{\partial}^b)=E_i^ag^{ij}E_j^b$ and $(g_{ab})=(g^{ab})^{-1}$. Then, it is obtained that:
$$J(\bar{\partial}^a)=g^{ab}\bar{E}_b^i\frac{\delta}{\delta x^i}$$
Now, if we let $\frac{\bar{\delta}}{\bar{\delta}x^a}=\bar{E}_a^i\frac{\delta}{\delta x^i}$, then we obtain:
$$J(\bar{\partial}^a)=g^{ab}\frac{\bar{\delta}}{\bar{\delta}x^b}\ ,\ \ \ J(\frac{\bar{\delta}}{\bar{\delta}x^a})=g_{ab}\bar{\partial}^b$$ and we can set a new local vector fields of $TT^*M$ as follows:
\begin{equation}~\label{basis1}
TT^*M=<\xi,\frac{\bar{\delta}}{\bar{\delta}x^a},C^*,\bar{\partial}^a>
\end{equation}
and the Sasakian metric $G$ on $T^*M$ can be shown in these new local vector fields as follows:
\begin{equation}~\label{met.mat.}
G:=\left(%
\begin{array}{cccc}
  K^2 & 0 & 0 & 0\\
  0 & (g_{ab}) & 0 & 0\\
  0 & 0 & K^2 & 0\\
  0 & 0 & 0 & (g^{ab})\\
\end{array}%
\right)
\end{equation}
The Lie brackets of the vector fields in~(\ref{basis1}) are presented as follows:
$$
\left\{
\begin{array}{l}
\ [\frac{\bar{\delta}}{\bar{\delta}x^a},\frac{\bar{\delta}}{\bar{\delta}x^b}]=
(\frac{\bar{\delta}\bar{E}_b^i}{\bar{\delta}x^a}
-\frac{\bar{\delta}\bar{E}_a^i}{\bar{\delta}x^b})\frac{\delta}{\delta x^i}+\bar{E}_a^i\bar{E}_b^jR_{ijs}\partial^s,\\
\ [\frac{\bar{\delta}}{\bar{\delta}x^a},\bar{\partial}^b]=(\frac{\bar{\delta}E_i^b}{\bar{\delta}x^a}
-\bar{E}_a^kE_j^bN_{ki}^j)\partial^i-\bar{\partial}^b(\bar{E}_a^i)\frac{\delta}{\delta x^i},\\
\ [\bar{\partial}^a,\bar{\partial}^b]=(\partial^aE_i^b-\partial^bE_i^a)\partial^i,\\
\ [\frac{\bar{\delta}}{\bar{\delta}x^a},\xi]=(\bar{E}_a^jN_{jk}g^{ki}+p_j\frac{\bar{\delta}g^{ji}}{\bar{\delta}x^a}
-\xi(\bar{E}_a^i))\frac{\delta}{\delta x^i}\cr \hspace{1.8cm}+\bar{E}_a^ip_hg^{hj}R_{ijs}\partial^s,\\
\ [\bar{\partial}^a,\xi]=g^{ab}\frac{\bar{\delta}}{\bar{\delta}x^b}+(E_j^ap_kg^{kh}N_{hi}^j-\xi(E_i^a))
\partial^i,\\
\ [\frac{\bar{\delta}}{\bar{\delta}x^a},C^*]=-C^*(\bar{E}_a^i)\frac{\delta}{\delta x^i},\\
\ [\bar{\partial}^a,C^*]=\bar{\partial}^a-C^*(E_i^a)\partial^i,\\
\ [\xi,\xi]=[C^*,C^*]=[\xi,C^*]+\xi=0.
\end{array}
\right.
$$
Now, the Levi-Civita connection $\nabla$ on the Riemannian manifold $(T^*M,G)$ for the basis~(\ref{basis1}) is given by:
\begin{equation}~\label{levicivita}
\left\{
\begin{array}{l}
\nabla_{\frac{\bar{\delta}}{\bar{\delta}x^a}}\frac{\bar{\delta}}{\bar{\delta}x^b}=(\Gamma_{ab}^c+
\frac{\bar{\delta}\bar{E}_b^i}{\bar{\delta}x^a}E_i^c)\frac{\bar{\delta}}{\bar{\delta}x^c}+\frac{1}{2}
(R_{abc}-g_{abc})\bar{\partial}^c \cr +\frac{1}{2K^2}(p_ig^{ij}(\frac{\bar{\delta}E_j^c}{\bar{\delta}x^a}g_{cb}
+\frac{\bar{\delta}E_j^c}{\bar{\delta}x^b}g_{ca})-\bar{E}_a^i\bar{E}_b^j\xi(g_{ij}))\xi \cr
\nabla_{\bar{\partial}^a}\bar{\partial}^b=(\bar{\partial}^a(E_i^b)\bar{E}_d^i+\frac{1}{2}g_d^{ab})\bar{\partial}^d
-\frac{1}{K^2}g^{ab}C^* \\
-\frac{1}{2}E_i^aE_j^b\bar{E}_c^k(\frac{\delta g^{ij}}{\delta x^k}+N_{kh}^ig^{hj}+N_{kh}^jg^{hi})g^{cd}\frac{\bar{\delta}}{\bar{\delta}x^d} \cr
\nabla_{\frac{\bar{\delta}}{\bar{\delta}x^a}}\bar{\partial}^b=\frac{1}{2}(g_{ac}^b-R_{ac}^b)g^{cd}
\frac{\bar{\delta}}{\bar{\delta}x^d}-\frac{1}{2K^2}R_a^b\xi+\frac{\bar{\delta}E_i^b}{\bar{\delta}x^a}
\bar{E}_d^i\bar{\partial}^d \\
+\frac{1}{2}E_i^bE_j^c\bar{E}_a^k(\frac{\delta g^{ij}}{\delta x^k}+N_{kh}^jg^{hi}-N_{kh}^ig^{hj})g_{cd}\bar{\partial}^d \cr
\nabla_{\bar{\partial}^b}\frac{\bar{\delta}}{\bar{\delta}x^a}=\frac{1}{2}(g_{ac}^b-R_{ac}^b)g^{cd}
\frac{\bar{\delta}}{\bar{\delta}x^d}-\frac{1}{2K^2}(R_a^b+2\delta_a^b)\xi
\cr+\frac{1}{2}E_i^bE_j^c\bar{E}_a^k(\frac{\delta g^{ij}}{\delta x^k}+N_{kh}^jg^{hi}+N_{kh}^ig^{hj})g_{cd}\bar{\partial}^d+\bar{\partial}^b(\bar{E}_a^k)E_k^d
\frac{\bar{\delta}}{\bar{\delta}x^d}
\end{array}
\right.
\end{equation}
\begin{equation}~\label{levicivita1}
\left\{
\begin{array}{l}
\nabla_{\frac{\bar{\delta}}{\bar{\delta}x^a}}\xi=\frac{1}{2}(\bar{E}_a^i\bar{E}_c^j\xi(g_{ij})+p_j
\frac{\bar{\delta}g^{ji}}{\bar{\delta}x^c}\bar{E}_a^kg_{ik}+\bar{E}_a^i\bar{E}_c^hN_{ih}\cr-p_jg^{ji}
\frac{\bar{\delta}E_i^b}{\bar{\delta}x^a}g_{bc})g^{cd}\frac{\bar{\delta}}{\bar{\delta}x^d}+\frac{1}{2}R_{ad}
\bar{\partial}^d \cr
\nabla_{\xi}\frac{\bar{\delta}}{\bar{\delta}x^a}=\frac{1}{2}(\bar{E}_a^i\bar{E}_c^j\xi(g_{ij})+p_j
\frac{\bar{\delta}g^{ji}}{\bar{\delta}x^c}\bar{E}_a^kg_{ik}+\bar{E}_a^i\bar{E}_c^hN_{ih}\cr+p_jg^{ji}
\frac{\bar{\delta}E_i^b}{\bar{\delta}x^a}g_{bc})g^{cd}\frac{\bar{\delta}}{\bar{\delta}x^d}+\xi(\bar{E}_a^i)E_i^d
\frac{\bar{\delta}}{\bar{\delta}x^d}-\frac{1}{2}R_{ad}\bar{\partial}^d \cr
\nabla_{\bar{\partial}^a}\xi=(g^{ad}+\frac{1}{2}R_{bc}g^{ba}g^{cd})\frac{\bar{\delta}}{\bar{\delta}x^d} \cr
\nabla_{\xi}\bar{\partial}^a=\frac{1}{2}R_{bc}g^{ba}g^{cd}\frac{\bar{\delta}}{\bar{\delta}x^d}+\bar{E}_d^i(\xi(E_i^a)
-p_kg^{kh}E_s^aN_{hi}^s)\bar{\partial}^d
\end{array}
\right.
\end{equation}
\begin{equation}~\label{levicivita2}
\left\{
\begin{array}{l}
\nabla_{\frac{\bar{\delta}}{\bar{\delta}x^a}}C^*=\nabla_{C^*}\frac{\bar{\delta}}{\bar{\delta}x^a}-C^*(\bar{E}_a^i)E_i^d
\frac{\bar{\delta}}{\bar{\delta}x^d}=0 \cr
\nabla_{\bar{\partial}^a}C^*-\bar{\partial}^a=\nabla_{C^*}\bar{\partial}^a-C^*(E_i^a)\bar{E}_d^i\bar{\partial}^d=0 \cr
\nabla_{\xi}C^*=\nabla_{C^*}\xi-\xi=\nabla_{\xi}\xi=\nabla_{C^*}C^*-C^*=0
\end{array}
\right.
\end{equation}
where
$$R_{abc}=\bar{E}_a^i\bar{E}_b^j\bar{E}_c^kR_{ijk}=g_{cd}R_{ab}^d,\ \ \ R_a^b=R_{ac}g^{cb}=\bar{E}_a^i\bar{E}_c^jR_{ij}g^{cb}$$
$$g_{abc}=\bar{E}_a^i\bar{E}_b^j\bar{E}_c^kg_{ijk}=g_{cd}g_{ab}^d=g_{ad}g_{be}g_c^{de}$$
$$\Gamma_{ab}^c=\bar{E}_a^i\bar{E}_b^jE_k^c\Gamma_{ij}^k,\ \ \ N_{ab}^c=\bar{E}_a^i\bar{E}_b^jE_k^cN_{ij}^k$$
The c-\emph{indicatrix} bundle of $T^*M$ denoted by $I^*\!M(c)$ is defined as follows:
$$I^*\!M(c)=\{(x,p)\in T^*M | K(x,p)=c>0\}.$$
Where for each $c\in\textbf{R}^+$ normal vector field to $I^*\!M(c)$ is given by:
\begin{equation}~\label{grad}
grad(K)=\frac{p_i}{c}\partial^i=\frac{1}{c}C^*
\end{equation}
Therefore,
\begin{equation}~\label{basis3}
T(I^*\!M)=<\xi,\frac{\bar{\delta}}{\bar{\delta}x^a},\bar{\partial}^a>
\end{equation}
In following, the Levi-Civita connection and metric on $c$-indicatrix bundles are denoted by $\bar{\nabla}$ and $\bar{G}$, respectively, which $\bar{G}$ is the restriction of metric~(\ref{met.mat.}). In order to compute the components of Levi-Civita connection $\bar{\nabla}$ on indicatrix bundle $I^*\!M(c)$ for the basis~(\ref{basis3}) the \emph{Gauss Formula}~\cite{lee}:
$$\nabla_XY=\bar{\nabla}_XY+H(X,Y)$$
where $H$ is the \emph{second fundamental form} of $I^*\!M(c)$ is needed. It is obvious that all components in~(\ref{levicivita})--(\ref{levicivita2}) except $\nabla_{\bar{\partial}^a}\bar{\partial}^b$ are tangent to $I^*\!M(c)$. Therefore, $\bar{\nabla}$ is equal to $\nabla$ for the other components of~(\ref{levicivita})--(\ref{levicivita2}). Moreover, the curvature tensor $R$ of $\nabla$ defined by $R(X,Y)Z=\nabla_X\nabla_YZ-\nabla_Y\nabla_XZ-\nabla_{[X,Y]}Z$
is related to the curvature tensor $\bar{R}$ of $\bar{\nabla}$ in following equations:
$$
\left\{
\begin{array}{l}
R(\frac{\bar{\delta}}{\bar{\delta}x^a},\frac{\bar{\delta}}{\bar{\delta}x^b})\bar{\partial}^c
=\bar{R}(\frac{\bar{\delta}}{\bar{\delta}x^a},\frac{\bar{\delta}}{\bar{\delta}x^b})
\bar{\partial}^c+\frac{1}{K^2}R_{abe}g^{ec}C^*\\
R(\frac{\bar{\delta}}{\bar{\delta}x^a},\bar{\partial}^b)\frac{\bar{\delta}}{\bar{\delta}x^c}
=\bar{R}(\frac{\bar{\delta}}{\bar{\delta}x^a},\bar{\partial}^b)
\frac{\bar{\delta}}{\bar{\delta}x^c}+\frac{1}{2K^2}(R_{acd}-g_{acd})g^{db}C^*\\
R(\bar{\partial}^a,\bar{\partial}^b)
\bar{\partial}^c
=\bar{R}(\bar{\partial}^a,\bar{\partial}^b)
\bar{\partial}^c+\frac{1}{K^2}(g^{ac}\bar{\partial}^b-
g^{bc}\bar{\partial}^a)\\
R(\frac{\bar{\delta}}{\bar{\delta}x^a},\bar{\partial}^b)
\bar{\partial}^c
=\bar{R}(\frac{\bar{\delta}}{\bar{\delta}x^a},\bar{\partial}^b)
\bar{\partial}^c-\frac{1}{2K^2}\bar{E}_a^kE_i^bE_j^c(\frac{\delta g^{ij}}{\delta x^k}+N_{kh}^ig^{hj}\cr \hspace{2.5cm}+N_{kh}^jg^{hi})C^*\\
R(\frac{\bar{\delta}}{\bar{\delta}x^a},\bar{\partial}^b)\xi=
\bar{R}(\frac{\bar{\delta}}{\bar{\delta}x^a},\bar{\partial}^b)\xi+\frac{1}{2K^2}R_{ad}g^{db}C^*\\
R(\bar{\partial}^a,\xi)\frac{\bar{\delta}}{\bar{\delta}x^b}=
\bar{R}(\bar{\partial}^a,\xi)\frac{\bar{\delta}}{\bar{\delta}x^b}+\frac{1}{2K^2}R_{bd}g^{da}C^*\\
R(\frac{\bar{\delta}}{\bar{\delta}x^a},\xi)\bar{\partial}^b=
\bar{R}(\frac{\bar{\delta}}{\bar{\delta}x^a},\xi)\bar{\partial}^b+\frac{1}{K^2}R_{ad}g^{db}C^*
\end{array}
\right.
$$
for the other combinations of $\frac{\bar{\delta}}{\bar{\delta}x^a},\bar{\partial}^a$ and $\xi$, tensor fields $R$ and $\bar{R}$ coincide with each other.

\section{Foliations on $T^*M$}
Let $(M,K)$ be a Cartan space defined in Section 2. The purpose of this section is to prove some results of~\cite{bejancu} on six natural foliations in the cotangent bundle of a Cartan space. It is interesting to see, that a study of these foliations provides important information on the geometry of the Cartan spaces themselves. These six foliations are presented as follows:
\begin{enumerate}
\item{$C^*$: Liouville-Hamilton vector field.}
\item{$\xi$: defined by $\xi:=JC^*$.}
\item{$C^*\oplus\xi$}.
\item{$VT^*M$: defined by $VT^*M=<\partial^1,...,\partial^n>$.}
\item{$V'T^*M$: which is perpendicular to $C^*$ in $VT^*M$ with respect to the metric $G$ defined in~(\ref{metric}).}
\item{$V^{\perp}T^*M$: which is perpendicular to $C^*$ in $TT^*M$ with respect to the metric $G$ defined in~(\ref{metric}).}
\end{enumerate}
\begin{theorem}~\label{tot geo 2}
$C^*$, $\xi$ and $C^*\oplus\xi$ are three totally geodesic foliations on $(T^*M,G)$.
\end{theorem}
\begin{proof}
According to~(\ref{levicivita2})
$$\nabla_{\xi}C^*=\nabla_{C^*}\xi-\xi=\nabla_{\xi}\xi=\nabla_{C^*}C^*-C^*=0$$
which this shows they are totally geodesic. \end{proof}


\begin{theorem}~\label{bundle like}
The lifted metric $G$ defined in~(\ref{metric}) is bundle-like for the vertical foliation $VT^*M$ if and only if $(g^{ij})$ is a Riemannian metric on $M$.
\end{theorem}
\begin{proof}
With respect to bundle-like condition (see~\cite{bejancu book}), $G$ is bundle-like for $VT^*M$ if and only if
$$G(\nabla_{\frac{\delta}{\delta x^i}}\frac{\delta}{\delta x^j}+\nabla_{\frac{\delta}{\delta x^j}}\frac{\delta}{\delta x^i},\partial^k)=0\ \ \ \forall i,j,k=1,...,n$$
Then, by Theorem~\ref{levi1}, it is deduced that $G$ is bundle-like for $VT^*M$ if and only if $g_{ijk}=0$. This completes the proof. \end{proof}

By using the basis~(\ref{basis1}) and~(\ref{levicivita})--(\ref{levicivita2}) the following theorem is obtained:

\begin{theorem}~\label{bundle like1}
The lifted metric $G$ defined in~(\ref{metric}) is bundle-like for the foliation $V'T^*M$ if and only if $(g^{ij})$ is a Riemannian metric on $M$.
\end{theorem}
\begin{proof}
From~(\ref{levicivita})-(\ref{levicivita2}), it is obtained:
$$G(\nabla_{\xi}\frac{\bar{\delta}}{\bar{\delta}x^a}+\nabla_{\frac{\bar{\delta}}{\bar{\delta}x^a}}\xi,\bar{\partial}^c)=
G(\nabla_{C^*}\frac{\bar{\delta}}{\bar{\delta}x^a}+\nabla_{\frac{\bar{\delta}}{\bar{\delta}x^a}}C^*,\bar{\partial}^c)$$
$$=G(\nabla_{\xi}C^*+\nabla_{C^*}\xi,\bar{\partial}^c)=G(2\nabla_{C^*}C^*,\bar{\partial}^c)
=G(2\nabla_{\xi}\xi,\bar{\partial}^c)=0$$
and
$$G(\nabla_{\frac{\bar{\delta}}{\bar{\delta}x^a}}\frac{\bar{\delta}}{\bar{\delta}x^b}
+\nabla_{\frac{\bar{\delta}}{\bar{\delta}x^b}}\frac{\bar{\delta}}{\bar{\delta}x^a},\bar{\partial}^c)=-g_{abd}g^{dc}$$
Therefore, the lifted metric $G$ defined in~(\ref{metric}) is bundle-like for the foliation $V'T^*M$ if and only if $g_{abd}=0$ and it is equivalent to $g^{ijk}=0$. This completes the proof. \end{proof}

\begin{theorem}~\label{C^*}
The lifted metric $G$ defined in~(\ref{metric}) is not bundle-like for the foliation $C^*$, or $C^*$ is not a Killing vector field of metric $G$.
\end{theorem}
\begin{proof}
From~(\ref{levicivita}), it is obtain that:
$$\mathcal{L}_{C^*}G(\bar{\partial}^a,\bar{\partial}^b)=G(\nabla_{\bar{\partial}^a}C^*,\bar{\partial}^b)
+G(\bar{\partial}^a,\nabla_{\bar{\partial}^b}C^*)$$
$$=-G(C^*,\nabla_{\bar{\partial}^a}\bar{\partial}^b)-G(\nabla_{\bar{\partial}^b}\bar{\partial}^a,C^*)=2g^{ab}$$
Therefore, $C^*$ is not a Killing vector field and $G$ is not bundle-like for $C^*$. \end{proof}

\begin{theorem}~\label{tot geo 1}
The vertical distribution $VT^*M$ is totally geodesic if and only if the following holds:
$$\frac{\delta g^{ij}}{\delta x^k}+N_{ks}^ig^{sj}+N_{ks}^jg^{si}=0$$
\end{theorem}
\begin{proof}
It is a straight conclusion of Theorem~\ref{levi1}. \end{proof}

\begin{theorem}~\label{tot. geo.}
The foliations $V'T^*M$ and $V^{\perp}T^*M$ are not totally geodesic in Riemannian manifold $(T^*M,G)$.
\end{theorem}
\begin{proof}
From~(\ref{levicivita}), we obtain
$H(\bar{\partial}^a,\bar{\partial}^b)=-\frac{1}{K^2}g^{ab}C^*$
which it cannot be vanish. \end{proof}

\begin{theorem}~\label{umblical}
The foliation $V'T^*M$ is a totally umbilical foliation of $TT^*M$.
\end{theorem}
\begin{proof}
From~(\ref{levicivita}), we obtain
$$H(\bar{\partial}^a,\bar{\partial}^b)=-\frac{1}{K^2}g^{ab}C^*$$
that it shows $V'T^*M$ is totally umbilical in $TT^*M$. \end{proof}

\begin{theorem}~\label{indicat}
The tangent bundles $TI^*\!M(c)$ for all $c\in\textbf{R}^+$ is just the foliation determined by the integrable distribution $V^{\perp}T^*M$. Also, it can be shown that $C^*$ and $\xi$ are orthogonal and tangent to $I^*\!M(c)$, respectively.
\end{theorem}
\begin{proof}
In~(\ref{grad}), it is shown that $C^*$ is orthogonal to $TI^*\!M(c)$. Therefore, foliation of tangent bundles of level hypersurfaces of $K$ is equal to $V^{\perp}T^*M$. On other hand, it is easy to see that $\xi(K)=0$ and $\xi$ is tangent to level hypersurfaces of $K$. \end{proof}

From~(\ref{brackets}) the followings are obtained:
$$R_{ijk}g^{kh}p_h=0,\ \ \ R_{ijk}+R_{jki}+R_{kij}=0$$
Therefore, $R_{ij}$ defined by $p_hg^{hk}R_{ikj}$ is a symmetric tensor for indices $i$ and $j$. In~\cite{miron}, it is proved that Cartan space $(M,K)$ is of constant curvature $c$ if and only if the following holds:
\begin{equation}~\label{cur}
R_{ij}=cK^2h_{ij}
\end{equation}
where $h_{ij}$ is \emph{angular metric} tensor field of a Cartan space defined by:
\begin{equation}~\label{angular}
h_{ij}=g_{ij}-\frac{1}{K^2}p_ip_j
\end{equation}
Let $\wedge^*=(\wedge^*_{ij})$ given by
\begin{equation}~\label{ang cur 1}
\wedge^*_{ij}=R_{ij}+h_{ij}
\end{equation}
be a symmetric bilinear tensor field on $C^{\infty}(M)$-module $\Gamma(HT^*M)$, and call it \emph{angular curvature} of $M$.
\begin{lemma}~\label{ang cur}
For any $X\in\Gamma(HT^*M)$, the following holds:
$$\wedge^*(\xi,X)=0.$$
\end{lemma}
\begin{proof}
$$\wedge^*_{ij}p_hg^{hi}X^j=R_{ij}p_hg^{hi}X^j+g_{ij}p_hg^{hi}X^j-\frac{1}{K^2}p_ip_jp_hg^{hi}X^j$$
$$=0+p_jX^j-p_jX^j=0$$. \end{proof}

\begin{theorem}~\label{bundle and angu}
Let $(M,K)$ be a Cartan space and $I^*\!M(c)$ be a $c$-indicatrix over $M$. Then, metric $G$ is bundle-like on $I^*\!M(c)$ for $\xi$ if and only if $\wedge^*=0$ on $I^*\!M(c)$.
\end{theorem}
\begin{proof}
$\xi$ is bundle-like on $I^*\!M(c)$ if and only if the following holds:
$$G(\nabla_XY,\xi)+G(\nabla_YX,\xi)=0$$
where $X=X^i\frac{\delta}{\delta x^i}+\bar{X}_j\partial^j$, $Y=Y^i\frac{\delta}{\delta x^i}+\bar{Y}_j\partial^j$, $X^ip_i=Y^ip_i=0$ and $\bar{X}_jp_ig^{ij}=\bar{Y}_jp_ig^{ij}=0$. By help of Theorem~\ref{levi1}, it can be obtain that:
$$G(\nabla_{\frac{\delta}{\delta x^i}}\xi,\frac{\delta}{\delta x^j})=G(\nabla_{\partial^i}\xi,\partial^j)=0$$
and
$$G(\nabla_{\partial^i}\xi,\frac{\delta}{\delta x^j})=G(\nabla_{\frac{\delta}{\delta x^j}}\xi,\partial^i)+\delta_j^i=\delta_j^i+\frac{1}{2}R_{js}g^{si}$$
Therefore, we obtain that:
$$0=G(\nabla_XY,\xi)+G(\nabla_YX,\xi)=-G(\nabla_X\xi,Y)-G(\nabla_Y\xi,X)$$
$$=-\bar{X}_jY^i(\delta_i^j+R_{is}g^{sj})-\bar{Y}_jX^i(\delta_i^j+R_{is}g^{sj})$$
$$\Longleftrightarrow (\bar{X}_jY^i+\bar{Y}_jX^i)(\delta_i^j+R_{is}g^{sj})=0$$
$$\Longleftrightarrow (\bar{X}_jY^i+\bar{Y}_jX^i)(\delta_i^j+R_{is}g^{sj}-\frac{1}{K^2}p_ip^j)=0$$
$$\Longleftrightarrow g_{ij}-\frac{1}{K^2}p_ip_j+R_{ij}=0 \Longleftrightarrow \wedge^*_{ij}=h_{ij}+R_{ij}=0$$. \end{proof}

\begin{theorem}~\label{Killing}
Let $(M,K)$ be a Cartan space. Then, $\xi$ is a Killing vector field on $I^*\!M(c)$ if and only if $\wedge^*=0$ on $I^*\!M(c)$.
\end{theorem}
\begin{proof}
Here, it is shown that $\xi$ is a Killing vector field on $I^*\!M(c)$ if and only if metric $G$ is bundle-like on $I^*\!M(c)$ for $\xi$. Then by Theorem~\ref{bundle and angu} the proof is complete. It is obvious that if $\xi$ is a Killing vector field then $G$ is bundle like on $I^*\!M(c)$ for $\xi$. The converse state is true if $G$ be bundle like on $I^*\!M(c)$ for $\xi$ and the followings hold:
$$G(\nabla_{\xi}\xi,\xi)=G(\nabla_{\xi}\xi,X)+G(\nabla_X\xi,\xi)=0$$
for all $X\in\Gamma(I^*\!M(c))$ where $G(X,\xi)=0$. Since $\nabla_{\xi}\xi=0$, it is enough to show that $G(\nabla_{\frac{\bar{\delta}}{\bar{\delta}x^a}}\xi,\xi)=G(\nabla_{\bar{\partial}^a}\xi,\xi)=0$ which it is obvious from~(\ref{levicivita1}) and it makes proof complete. \end{proof}

\begin{theorem}~\label{constant cur}
Let $(M,K)$ be a Cartan space. Then, $M$ is a Cartan space of negative constant curvature $k$ if and only if $\wedge^*$ be vanish on $I^*\!M(c)$, where $c^2=-\frac{1}{k}$.
\end{theorem}
\begin{proof}
Let $M$ be a Cartan space of negative constant curvature $k$, then by~(\ref{cur}) we obtain:
\begin{equation}~\label{111}
R_{ij}(x,p)=-h_{ij}(x,p)\ \ \ \ \ \forall (x,p)\in I^*\!M(c)
\end{equation}
Thus, by~(\ref{ang cur 1}) we have $\wedge^*=0$ on $I^*\!M(c)$. Conversely, suppose that $\wedge^*=0$ on $I^*\!M(c)$. To complete the proof it is enough to show that~(\ref{111}) is valid for $(x,p)\in T^*M\setminus I^*\!M(c)$. For all $(x,p)\in T^*M$ there exists a constant $\bar{c}$ which $(x,\frac{p}{\bar{c}})\in I^*\!M(c)$. By definition $R_{ij}$ and equations~(\ref{brackets}) and~(\ref{angular}), it is easy to see that $R_{ij}$ and $h_{ij}$ are homogeneous of degree two and zero, respectively. Thus:
$$R_{ij}(x,\frac{p}{\bar{c}})=-h_{ij}(x,\frac{p}{\bar{c}}) \Longleftrightarrow R_{ij}(x,p)=-\bar{c}^2h_{ij}(x,p)$$
$$\Longleftrightarrow R_{ij}(x,p)=kK^2(x,p)h_{ij}(x,p) \ \ \ \forall (x,p)\in T^*M$$
and this completes the proof. \end{proof}

Combining Theorems~\ref{bundle and angu},~\ref{Killing} and~\ref{constant cur} the following is obtained:
\begin{theorem}~\label{final}
Let $(M,K)$ be a Cartan space, and $k<0<c$ two constant where $-c^2k=1$. Then the followings are equivalent:
\begin{enumerate}
\item{$M$ is a Cartan space of constant curvature k.}
\item{$G$ is bundle-like on $I^*\!M(c)$ for $\xi$.}
\item{$\xi$ is a Killing vector field on $I^*\!M(c)$.}
\item{Angular curvature $\wedge^*$ is vanish on $I^*\!M(c)$.}
\end{enumerate}
\end{theorem}

\section{The contact structure on indicatrix bundle of a Cartan space}
In the first part of this section, it is shown that each $I^*\!M(c)$ naturally has contact structure. Then in the next part, it is proved that this contact structure cannot be a Sasakian one.

The (1,1)-tensor field $\varphi$ is set on the indicatrix bundle $I^*\!M(1)$ as follows:
\begin{equation}~\label{phi}
\varphi:=-J|_{D} \ ,\ \ \ \ \varphi(\xi)=0
\end{equation}
where $D=\{X\in TT^*M\ |\ G(X,\xi)=G(X,C^*)=0\}$. The distribution $D$ is called \emph{contact distribution} in contact manifolds. For more on Riemannian geometry adopted to such distributions see \cite{Blair}. Also, notation $\bar{G}$ is used to restrict metric $G$ to the c-indicatrix bundle. It is obvious that the dual 1-form $\xi$ with respect to the metric $G$ is lioville 1-form $\omega$ defined in~(\ref{geo object}) in Cartan space $(M,K)$. Now, the following theorem can be expressed:
\begin{theorem}~\label{contact}
Let the 4-tuple $(\varphi,\omega,\xi,\bar{G})$ be defined as above. Then 1-indicatrix bundle of a Cartan space with $(\varphi,\omega,\xi,\bar{G})$ is a contact manifold.
\end{theorem}
\begin{proof}
The compatibility of $\varphi$ and the metric $\bar{G}$ is equivalent to compatibility of $J$ and $G$. Also, The conditions
$$\omega\circ\varphi=0\ ,\ \ \ \ \varphi(\xi)=0\ ,\ \ \ \ \varphi^2=-I+\xi\otimes\omega$$
are easy to be proved by considering Eq.~(\ref{phi}) and definitions given in above. To complete the proof, the condition $d\omega(X,Y)=\bar{G}(X,\varphi Y)$ for the vector fields $X,Y\in\Gamma(TI^*\!M)$ needs to be checked. By calculating $d\omega$ the following can be obtained:\\
$$d\omega=N_{ij}dx^j\wedge dx^i+\delta p_i\wedge dx^i=\delta p_i\wedge dx^i$$
Also,
$$G(\frac{\delta}{\delta x^i},J\frac{\delta}{\delta x^j})=G(\partial^i,J\partial^j)=0$$
and
$$G(\frac{\delta}{\delta x^i},J\partial^j)=-G(\partial^i,J\frac{\delta}{\delta x^j})=\delta_i^j$$
$$\Longrightarrow \ \ G(.,J.)=dx^i\wedge\delta p_i$$
$$\Longrightarrow \ \ d\omega(X,Y)=-G(X,JY)$$
Since $\bar{G}$ is the restriction of $G$ to the indicatrix and $d\omega(\xi,.)=0$ thus
$$d\omega(X,Y)=\bar{G}(X,\varphi Y)\ \ \ \ \ \forall X,Y\in\Gamma D$$
So, $I^*\!M(1)$ with $(\varphi,\omega,\xi,\bar{G})$ is a contact manifold~\cite{Blair}. \end{proof}


Now, we show that this contact structure cannot be a Sasakian one. First, let $(M,\varphi,\eta,\xi,g)$ be a contact Riemannian manifold. In~\cite{bejancu2}, it was proved that $M$ is Sasakain manifold if and only if
$$(\tilde{\nabla}_X\varphi)Y=0 \ \ \ \ \forall X,Y\in\Gamma(TM)$$
where
$$\tilde{\nabla}_XY=\nabla_XY-\eta(X)\nabla_Y\xi-\eta(Y)\nabla_X\xi+(d\eta+\frac{1}{2}(\mathcal{L}_{\xi}g))(X,Y)\xi$$
and $\nabla$ is Levi-Civita connection on $(M,g)$.

Since the indicatrix bundle has the contact metric structure in Cartan spaces by Theorem~\ref{contact}, the following question cross our mind that "\emph{Can the indicatrix bundle in a Cartan space be a Sasakian manifold?}". First, the following Lemma is proved in order to reduce the number of calculations.
\begin{lemma}\label{redu}
If $(M,\varphi,\omega,\xi,g)$ be a contact metric manifold with contact distribution $D$, then $M$ is Sasakian manifold if and only if:
$$(\tilde{\nabla}_X\varphi)Y=0 \ \ \ \ \forall X,Y\in\Gamma(D)$$
\end{lemma}
\begin{proof}
For all $\bar{X}\in\Gamma(TM)$, they can be written in the form $X+f\xi$ where $X\in\Gamma D$, $f\in C^{\infty}(M)$ and $\xi$ is Reeb vector field of the contact structure on $M$. Therefore:
$$(\tilde{\nabla}_{\bar{X}}\varphi)\bar{Y}=(\tilde{\nabla}_{X+f\xi}\varphi)(Y+g\xi)=
(\tilde{\nabla}_X\varphi)Y+(\tilde{\nabla}_{f\xi}\varphi)Y+(\tilde{\nabla}_X\varphi)g\xi$$
$$+(\tilde{\nabla}_{f\xi}\varphi)g\xi=(\tilde{\nabla}_X\varphi)Y+f(\tilde{\nabla}_{\xi}\varphi Y-\varphi\tilde{\nabla}_{\xi}Y)+\tilde{\nabla}_X\varphi(g\xi)-\varphi(\tilde{\nabla}_Xg\xi)$$
$$+f(\tilde{\nabla}_{\xi}\varphi{g\xi}-\varphi\tilde{\nabla}_{\xi}g\xi)=(\tilde{\nabla}_X\varphi)Y$$
\noindent The lemma is proved using Theorem 3.2 in~\cite{bejancu2} and the last equation. \end{proof}

Now, the following theorem can be expressed:
\begin{theorem}~\label{Sasaki}
Let $(M,K)$ be a Cartan space. Then, indicatrix bundle $I^*\!M(1)$ with contact structure $(I^*\!M(1),\varphi,\omega,\xi,\bar{G})$ can never be a Sasakian manifold.
\end{theorem}
\begin{proof}
From lemma~\ref{redu}, $I^*\!M(1)$ is a Sasakian manifold if and only if:
$$(\tilde{\nabla}_{\frac{\bar{\delta}}{\bar{\delta}x^a}}\varphi)\frac{\bar{\delta}}{\bar{\delta}x^b}= (\tilde{\nabla}_{\frac{\bar{\delta}}{\bar{\delta}x^a}}\varphi)\bar{\partial}^b= (\tilde{\nabla}_{\bar{\partial}^a}\varphi)\frac{\bar{\delta}}{\bar{\delta}x^b}= (\tilde{\nabla}_{\bar{\partial}^a}\varphi)\bar{\partial}^b=0$$
\noindent Using~(\ref{levicivita})--(\ref{levicivita2}), one of the components in above equations is
$$g_{ab}=0$$
\noindent which demonstrates a contradiction and shows that the indicatrix bundle cannot be a Sasakian manifold with contact structure $(\varphi,\omega,\xi,\bar{G})$. \end{proof}


\end{document}